\documentclass[11pt]{amsart}
\usepackage{mathtools}
\usepackage{graphicx}
\usepackage{color}
\usepackage[update,prepend]{epstopdf}
\usepackage{amsmath,amsfonts}
\usepackage{times}
\usepackage{cancel}
\usepackage[ansinew]{inputenc}
\usepackage{latexsym}
\numberwithin{equation}{section}

\newtheorem{lem}{Lemma}[section]
\newtheorem{dee}{Definition}[section]
\newtheorem{prop}{Proposition}

\newtheorem{thm}{Theorem}[section]

\let\cal=\mathcal

\let\bb=\mathbb

\let\eqre=\eqref

\def\endoc{\end{document}}
\def\beq*{\begin{equation*}}
 \def\eneq*{\end{equation*}}

\begin{document}
\date{\today}

\bibliographystyle{plain}

\title[Regularity of semigroup ... coupled elastic systems]{Regularity of the semigroups associated with some damped coupled elastic systems II: a nondegenerate fractional damping case}
\author[Ka\"is Ammari]{Ka\"is Ammari}
\address{UR Analysis and Control of PDEs, UR 13ES64, Department of Mathematics, Faculty of\\ \indent Sciences of Monastir, University of Monastir, 5019 Monastir, Tunisia}
\email{kais.ammari@fsm.rnu.tn}

\author[Farhat Shel]{Farhat Shel}
\address{UR Analysis and Control of PDEs, UR 13ES64, Department of Mathematics, Faculty of\\ \indent Sciences of Monastir, University of Monastir, 5019 Monastir, Tunisia}
\email{farhat.shel@fsm.rnu.tn}

\author[Louis Tebou]{Louis Tebou}
\address{Department of Mathematics and Statistics, Florida International University,
Modesto Maidique\\ \indent Campus, Miami, Florida 33199, USA}
\email{teboul@fiu.edu}
 
\begin{abstract}
In this paper, we examine regularity issues for two damped abstract elastic 
 systems; the damping and coupling involve fractional powers $\mu, \theta$, with $0 \leq \mu , \theta \leq 1$, of the principal operators. The matrix defining the coupling and damping is nondegenerate. This new work is a sequel to the degenerate case that we discussed recently in \cite{kfl}.  First, we prove that for $1/2 \leq \mu , \theta \leq 1$, the underlying semigroup is analytic. Next, we show that for $\min(\mu,\theta) \in (0,1/2)$, the semigroup is of certain Gevrey classes. Finally, some examples of application are provided.
\end{abstract}


\maketitle
\tableofcontents

\section{Introduction}  \label{sec1}
Let $H$ be a Hilbert space with inner product $\left\langle .,.\right\rangle $ and norm $|.|$. Let $A_1$ and $A_2$ be two self-adjoint on the Hilbert space $H$, strictly positive, with dense domains $D(A_1)$ and $D(A_2)$ respectively. 
Let $B_1$ and $B_2$ be other self-adjoint, positive operators on the Hilbert space $H$, with dense domains $D(B_1)$ and $D(B_2)$ respectively, satisfying for some positive constants $\alpha_0$, $\alpha_1$, $\alpha_2$, $\beta_1$, $\beta_2$ and $\mu, \theta \in (0,1]$,
\begin{eqnarray*}
B_1\leq \alpha_0B_2, \label{b12'}\\
\beta_1A_2^\theta\leq B_2\leq \beta_2A_2^\theta,\label{b1'}\\
\beta_1A_2^\theta\leq B_2\leq \beta_2A_2^\theta, \label{b2'}
\end{eqnarray*}
in the sense that,
\begin{eqnarray}
|B_1^\frac{1}{2}u|^2\leq\alpha_0|B_2^\frac{1}{2}u|^2,\;\;\;u\in D(B_2^\frac{1}{2})\subseteq D(B_1^\frac{1}{2}),\label{b12}\\
\alpha_1|A_1^\frac{\mu}{2}u|^2\leq |B_1^\frac{1}{2}u|^2\leq \alpha_2|A_1^\frac{\mu}{2}u|^2,\;\;\;u\in D(B_1^\frac{1}{2})=D(A_1^\frac{\mu}{2}),\label{b1}\\
\beta_1|A_2^\frac{\theta}{2}u|^2\leq |B_2^\frac{1}{2}u|^2\leq \beta_2|A_2^\frac{\theta}{2}u|^2,\;\;\;u\in D(B_2^\frac{1}{2})=D(A_2^\frac{\theta}{2}).\label{b2}
\end{eqnarray}
Note that from (\ref{b12})-(\ref{b2}), one has $D(A_2^\frac{\theta}{2})\subseteq D(A_1^\frac{\mu}{2})$.
 
 Set $V_j=D(A_j^\frac{1}{2})$, $j=1,~2$. We assume that for each $j=1,2$, $V_j\hookrightarrow H\hookrightarrow V_j'$, each injection being dense and compact, where $V_j'$ denotes the topological dual of $V_j$.

\medskip

Let  $\alpha$ and $\gamma$ be positive constants, and let $\beta$ be a nonzero real constant.
\medskip

Consider the evolution system
\begin{equation}\label{e1}\begin{array}{lll} 
&y_{tt}+A_1y+\alpha B_1y_t+\beta B_1z_t=0\text{ in }(0,\infty)\\
&z_{tt}+A_2z+\beta B_1y_t+\gamma B_2z_t=0\text{ in }(0,\infty)\\
&y(0)=y^0\in V_1,\quad y_t(0)=y^1\in H,\quad z(0)=z^0\in V_2,\quad z_t(0)\in H.\end{array}\end{equation}
Introduce the Hilbert space ${\cal H}=V_1\times H\times V_2\times H$, over the field ${\bb C}$ of complex numbers, equipped with the norm
$$||Z||^2=|A_1^\frac{1}{2}u|^2+|v|^2+|A_2^\frac{1}{2}w|^2+|z|^2,\quad\forall Z=(u,v,w,z)\in{\cal H}.$$
Define the operator 
\begin{equation*}\label{eac}
{\cal A}_{\mu,\theta}=\left( {\begin{array}{cccc}  
0&I&0&0\\  
-A_1 & -\alpha B_1 &0&-\beta B_1  \\
0&0&0&I\\
0 &-\beta B_1&-A_2  & -\gamma B_2\\ 
\end{array}} \right)
\end{equation*}
 with domain
$$D({\cal A})=\Big\{(u,v,w,z)\in V_1\times V_1
\times
V_2\times V_2;~ A_1u+\alpha B_1v+\beta B_1z\in H,\text { and } A_2w+\beta B_1v+\gamma B_2z\in H \Big\}.$$
Then, by denoting $v=u_t$ and $z=w_t$, system (\ref{e1}) can be rewritten as an abstract linear evolution equation on the Hilbert space $\mathcal{H}$,
\begin{equation}
\left\{ 
\begin{array}{c}
\frac{dU}{dt} (t) =\mathcal{A}_{\mu,\theta}U,\;\;\;\;t\geq  0 \\ 
U(0)=(u^0,u^1,z^0, z^1)
\end{array}
\right.  \label{abs}
\end{equation} 
In the sequel we suppose that the constants $\alpha$, $\beta$, $\alpha_0$ and $\gamma$ satisfy the following inequality
\begin{equation}
\alpha \gamma >\beta^2\alpha_0 \label{cod-dis}
\end{equation}
We have,
\begin{prop}
The operator ${\cal A}_{\mu,\theta}$ is dissipative.
\end{prop}
\begin{proof}
One easily checks that for every $Z=(u,v,w,z)\in D({\cal A}_{\mu,\theta})$, 
\begin{equation*}\begin{array}{lll}
\Re({\cal A}_{\mu,\theta}Z,Z)&=-\alpha |B_1^\frac{1}{2}v|^2-2\beta\Re(B_1
^\frac{1}{2}v,B_1^\frac{1}{2}z)-\gamma|B_2^\frac{1}{2}z|^2\\
&\leq-\alpha |B_1^\frac{1}{2}v|^2+2|\beta||B_1^\frac{1}{2}v||B_1^\frac{1}{2}z|-\gamma|B_2^\frac{1}{2}z|^2\\
&\leq-\alpha |B_1^\frac{1}{2}v|^2+\frac{|\beta|}{c}|B_1^\frac{1}{2}v|^2+\alpha_0|\beta| c |B_2^\frac{1}{2}z|^2-\gamma|B_2^\frac{1}{2}z|^2,\end{array}
\end{equation*}
for some positive constant $c$.
 Then \begin{equation}
\Re({\cal A}_{\mu,\theta}Z,Z)\leq-(\alpha-\frac{|\beta|}{c}) |B_1^\frac{1}{2}v|^2-(\gamma-\alpha_0|\beta| c) |B_2^\frac{1}{2}z|.\label{dis}
\end{equation} 
Since $\alpha \gamma >\beta^2\alpha_0$, it is possible to choose $c$ such that  $k_1:=\alpha-\frac{|\beta|}{c}>0$ and $k_2:=\gamma-\alpha_0|\beta| c>0$ (equivalently $\frac{|\beta|}{\alpha}<c<\frac{\gamma}{\alpha_0|\beta|}$).
So that the operator ${\cal A}_{\mu,\theta}$ is dissipative.
\end{proof}

 Further, the operator ${\cal  A}_{\mu,\theta}$ is densely defined, so ${\cal A}_{\mu,\theta}$ is closable on ${\cal H}$. Therefore, the Lumer-Phillips Theorem shows that the operator ${\cal A}_{\mu,\theta}$ generates a strongly continuous semigroup of contractions $(S(t))_{t\geq0}$ on the Hilbert space ${\cal H}$, which leads to the well-posedness of the system (\ref{e1}). 

Moreover, as in \cite{kfl}, the operator ${\cal A}_{\mu,\theta}$ satisfies
\begin{equation}\label{sst}
i\mathbb{R}\subset \rho({\cal A}_{\mu,\theta})
\end{equation}
where $\rho({\cal A}_{\mu,\theta})$ denotes the resolvent set of ${\cal A}_{\mu,\theta}$.\\
As a consequence, the semigroup $e^{t\mathcal{A}}$ is strongly stable \cite{ArBa}.

Our main goal is to study some regularity properties for the solutions of the system (\ref{abs}).

Before going on, let us recall some definitions relevant to the regularity of $C_0$-semigroups.
\begin{dee}
 Let $T(t):=e^{t\mathcal{A}}$ be a $C_0$-semigroup on a Hilbert space $\mathcal{H}$.
\begin{enumerate}
\item The semigroup $T(t)$ is said to be \textit{analytic} if
\begin{itemize}
\item for some $\varphi \in (0,\frac{\pi}{2})$, $T(t)$  can be extended to $\Sigma_\varphi$, where
$$\Sigma_\varphi=\{0\}\cup \{\tau \in \mathbb{C} \;:\;|\mathrm{arg}(\tau)|<\varphi \},$$
so that for any $x\in \mathcal{H}$, $\tau\mapsto T(\tau)x$ is continuous on $\Sigma_\varphi$, and for each $\tau_1,\tau_2 \in \Sigma_\varphi$, $T(\tau_1+\tau_2)=T(\tau_1)T(\tau_2)$. 
\item The map $\tau \mapsto T(\tau)$ is analytic over $\Sigma_\varphi\setminus \{0\}$, in the sense of the uniform operator topology of $\mathcal{L}(\mathcal{H})$.
\end{itemize}
\item The semigroup $e^{t\mathcal{A}}$ is said to be \textit{differentiable} if for any $x\in\mathcal{H}$, $t\mapsto e^{t\mathcal{A}}x$ is differentiable on $(0,\infty)$.
\item The semigroup $e^{t\mathcal{A}}$ is said to be of \textit{Gevrey} class $\delta$ (with $\delta >1$) if $e^{t\mathcal{A}}$ is infinitely differentiable and for any compact subset $\mathcal{K}\subset (0,\infty)$ and any $\lambda>0$
, there exists a constant $C=C(\lambda,\mathcal{K})$ such that
$$\|\mathcal{A}^ne^{t\mathcal{A}}\|_{\mathcal{L}(\mathcal{H})}\leq C\lambda^n(n!)^\delta,\;\;\;\;\forall\;t\in\mathcal{K},\;n\geq 0.$$

\end{enumerate}
\end{dee}
In this paper we will use the following standard results to identify analytic or Gevrey class semigroups, based on the estimation for the resolvent of the generator of the semigroup.
\begin{lem}\cite{LiZh}\label{lem1}
Let $\mathcal{A}:D(\mathcal{A})\subset \mathcal{H}\rightarrow \mathcal{H}$ generate a $C_0$-semigroup of contractions $e^{t\mathcal{A}}$ on $\mathcal{H}$.
 Suppose that
\begin{equation}\label{asym}
i\mathbb{R}\subset \rho(\mathcal{A})
\end{equation}
where $\rho(A)$ denotes the resolvent of $\mathcal{A}$.

The semigroup $e^{t\mathcal{A}}$ is analytic if and only if
\begin{equation}
\label{ana}\limsup_{|\lambda|\rightarrow\infty}|\lambda|\|(i\lambda I-\mathcal{A})^{-1}\|_{\mathcal{L}(\mathcal{H})}<\infty.
\end{equation}
\end{lem}
\begin{lem}\cite{taylor}\label{lem2}
Let $\mathcal{A}:D(\mathcal{A})\subset \mathcal{H}\rightarrow \mathcal{H}$ generate a bounded $C_0$-semigroup  $e^{t\mathcal{A}}$ on $\mathcal{H}$. Suppose that $\mathcal{A}$ satisfies the following estimate, for some $0<\alpha<1$,

\begin{equation}
\label{ger}\limsup_{|\lambda| \rightarrow\infty}|\lambda|^\alpha \|(i\lambda I-\mathcal{A})^{-1}\|_{\mathcal{L}(\mathcal{H})}<\infty.
\end{equation}
Then $e^{t\mathcal{A}}$ is of Gevrey class $\delta$ for $t>0$, for every $\delta>\frac{1}{\alpha}$.
\end{lem}
In their work \cite{cru}, G. Chen and Russell considered the following system
$$y_{tt}+Ay+By_t=0\text{ in }(0,\infty)$$
where $A$ is a self-adjoint operator on the Hilbert space $H$, strictly positive, with domain dense in $H$ and $B$ is a positive self-adjoint operator with dense domain in $H$. Set $E=D(A^{\frac{1}{2}})\times H$. They show that if for every $\rho>0$, there exists $\varepsilon(\rho)>0$ such that
$$(2\rho-\varepsilon(\rho))A^{\frac{1}{2}}\leq B\leq (2\rho+\varepsilon(\rho))A^{\frac{1}{2}},$$then the underlying semigroup is analytic on $E$. \\ Then, they conjecture the analyticity of the semigroup provided that
$$\exists 0<\rho_1<\rho_2<\infty:\rho_1 A^{\frac{1}{2}}\leq B\leq \rho_2 A^{\frac{1}{2}},$$or else
$$\rho_1^2A\leq B^2\leq \rho_2^2 A.$$
Later on, S.P. Chen and Triggiani responded to those conjectures by proving that if 
$$\exists \mu\in(0,1],~~\exists 0<\rho_1<\rho_2<\infty:\rho_1 A^{\mu}\leq B\leq \rho_2 A^{\mu},$$
then the semigroup is
\begin{enumerate}
\item analytic for $\frac{1}{2}\leq\mu\leq1$,  but not analytic for $0<\mu<\frac{1}{2}$, \cite{ctra}
\item of Gevrey class $\delta$ for all $\delta>\frac{1}{2\mu}$ for  $0<\mu<\frac{1}{2}$, \cite{ctrg}
\end{enumerate}
and this on a wide range of energy spaces, not only $E$.  See also the works by Huang \cite {H2}, and Huang and Liu \cite{HuL}.\\
In particular, \cite {ctrg} generalizes the work of Taylor \cite{taylor} where the author discusses Gevrey semigroups, and illustrates his work with several examples including the case  $B=2\rho A^{\mu}$ for some positive constant $\rho$.
\\
Our purpose in this note is to examine the following questions:  Assume a coercive dissipative mechanism: $\alpha\gamma>\beta^2\alpha_0$. Do we have results similar to Chen-Triggiani \cite{ctra, ctrg}, namely, is the semigroup $(S(t))_{t\geq0}$ analytic for $\frac{1}{2}\leq\mu  , \theta\leq1$? And if $s:=\min(\mu,\theta)$ lies in $(0,\frac{1}{2})$, is the semigroup of Gevrey class $\delta$ for all $\delta>\frac{1}{2s}$?\\ Before answering those questions, we would like to mention the work \cite{ly} where the authors discuss an abstract evolutionary system of the form $$\dot Z=\mathcal A Z$$ with $\mathcal A$ given by 
$$ \mathcal A=\begin{pmatrix}-A_0&B\\C&-A_1\end{pmatrix},$$ where $C=-B^*$. They establish sufficient conditions on the operators $A_0$, $A_1$, $B$ and $C$ for the operator $\mathcal A$ to generate an exponentially stable, analytic, differentiable or Gevrey class semigroup. Their results apply to many dynamical systems. However, the condition $C=-B^*$ excludes the abstract system considered in this note, where we have $B=C$, if we use the notations in \cite{ly}.

\medskip

\section{Analyticity: case $\frac{1}{2} \leq \mu , \theta \leq 1$}
\begin{thm}\label{th1}
If $\frac{1}{2}\leq \mu ,\theta \leq 1$ the associated semigroup $(S(t))_{t\geq 0}$ is analytic.
\end{thm}
\begin{proof}
Since $i\mathbb{R}\subset \rho({\cal A}_{\mu,\theta}),$ then, using Lemma \ref{lem1}, it suffices to show (\ref{ana}).
Suppose the result is false, then there exist a sequence $\lambda_n$ of real numbers with $|\lambda_n|$ going to $\infty$ as $n\rightarrow \infty$, and a sequence $Z_n=(u_n,v_n,w_n,z_n)$ in $D(\mathcal{A}_{\mu,\theta})$, with $||Z_n||=1$ and
\begin{equation}\label{lim2}
\lim_{|\lambda_n|\to\infty} |\lambda_n|^{-1}||(i\lambda_n I - {\cal A}_{\mu,\theta})Z_n||=0.
\end{equation}
First, taking  the real part of the inner product of $|\lambda_n|^{-1}\left(i\lambda_n I - {\cal A}_{\mu,\theta}\right) Z_n$ with $Z_n$, then  using (\ref{lim2}), (\ref{dis}) and the condition (\ref{cod-dis}), we obtain
\begin{equation}\label{diss}
|\lambda_n|^{-{\frac{1}{2}}}|B_1^{\frac{1}{2}}v_n|=o(1),\;\;\;\;\;\;|\lambda_n|^{-{1\over2}}|B_2^\frac{1}{2}z_n|=o(1)\;\;\;\; \text{and}\;\;\;\;|\lambda_n|^{-{1\over2}}|B_1^\frac{1}{2}z_n|=o(1).
\end{equation}
Now equation \eqref{lim2} can be rewritten explicitly as follows
\begin{eqnarray}
|\lambda_n|^{-1}\left(i\lambda_nA_1^{1\over2}u_n-A_1^{1\over2}v_n\right)=o(1), \label{eq1}\\
|\lambda_n|^{-1}\left(i\lambda_nv_n+A_1u_n+\alpha B_1 v_n + \beta B_1 z_n\right)=o(1),\label{eq2}\\
|\lambda_n|^{-1}\left(i\lambda_nA_2^{1\over2}w_n-A_2^{1\over2}z_n\right)=o(1),\label{eq3}\\
|\lambda_n|^{-1}\left(i\lambda_nz_n+ \beta B_1u_n+ A_2w_n+\gamma B_2 z_n\right)=o(1).\label{eq4}
\end{eqnarray}
Taking the inner product of \eqref{eq1} and \eqref{eq2} with $A_1^{1\over2}u_n$ and $v_n$ respectively, and \eqref{eq3} and \eqref{eq4} with $A_2^{1\over2}w_n$ and $z_n$ respectively, one obtains
\begin{eqnarray}
|\lambda_n|^{-1}\left( i\lambda_n|A_1^{1\over2}u_n|^2-\left\langle v_n,A_1u_n\right\rangle\right)  =o(1), \label{eq11}\\
|\lambda_n|^{-1}\left( i\lambda_n|v_n|^2+\left\langle A_1u_n,v_n\right\rangle+\alpha|B_1^{\frac{1}{2}}v_n|^2
+\beta\left\langle B_1z_n,v_n\right\rangle\right)  =o(1),\label{eq21}\\
|\lambda_n|^{-1}\left( i\lambda_n|A_2^{1\over2}w_n|^2-\left\langle z_n,A_2w_n\right\rangle\right)  =o(1), \label{eq31}\\
|\lambda_n|^{-1}\left( i\lambda_n|z_n|^2+\left\langle A_2w_n,z_n\right\rangle+\beta\left\langle B_1v_n,z_n\right\rangle+\gamma |B_2^{\frac{1}{2}}z_n|^2\right)  =o(1). \label{eq41}
\end{eqnarray}
Taking into account estimations \eqref{diss}, the equations \eqref{eq21} and \eqref{eq41} can be rewritten as follows 
\begin{eqnarray}
|\lambda_n|^{-1}\left( i\lambda_n|v_n|^2+\left\langle A_1u_n,v_n\right\rangle\right)
  =o(1),\label{eq211}\\
|\lambda_n|^{-1}\left( i\lambda_n|z_n|^2+\left\langle A_2w_n,z_n\right\rangle \right) =o(1)\label{eq411}.
\end{eqnarray}
Combining conjugates of \eqref{eq211} and \eqref{eq411} with \eqref{eq11} and \eqref{eq31}, one get after dividing by $\lambda_n|\lambda_n|^{-1}$
\begin{equation*}
 |A_1^{1\over2}u_n|^2+|A_2^{1\over2}w_n|^2-|v_n|^2-|z_n|^2  =o(1),
\end{equation*}
which, thanks to $||Z_n||=1$, leads to
\begin{equation}
 |v_n|^2+|z_n|^2-\frac{1}{2} =o(1).
\label{eqq}
\end{equation}
We will prove that $|v_n|=o(1)$ and $|z_n|=o(1)$ which contradicts (\ref{eqq}). For this, dividing \eqref{eq211} and \eqref{eq411} by $|\lambda_n|^{-1}\lambda_n$ to get
\begin{equation}
 i|v_n|^2+\frac{1}{\lambda_n} \left\langle A_1u_n,v_n\right\rangle =o(1),\label{eQQ1}
\end{equation}  
and
\begin{equation}
 i|z_n|^2+\frac{1}{\lambda_n} \left\langle A_2w_n,z_n\right\rangle =o(1)\label{eQQ2}.
\end{equation} 
Thus, it suffices to prove that $\frac{1}{\lambda_n} \left\langle A_1u_n,v_n\right\rangle =o(1)$ and $\frac{1}{\lambda_n}\left\langle A_2w_n,z_n\right\rangle =o(1)$.

\medskip

To start, we have the following estimate by applying the Cauchy-Schwarz inequality,
\begin{equation}
\left|\frac{1}{\lambda_n} \left\langle A_1u_n,v_n\right\rangle \right| \leq \frac{|A_1^{1-\frac{\mu}{2}}u_n|}{|\lambda_n|^{1\over2}}\frac{|A_1^{\mu \over2}v_n|}{|\lambda_n|^{1\over2}}.\label{csch1}
\end{equation}

\medskip

Note that  $A_1^{1-\mu}u_n$ is bounded, since $A_1^{1-\mu}u_n=A_1^{\frac{1}{2}-\mu}\left( A_1^{\frac{1}{2}}u_n\right)$ and $A_1^{\frac{1}{2}-\mu}$ is a bounded operator and $A_1^{\frac{1}{2}}u_n$ is also bonded. Then  taking the inner product of (\ref{eq2}) with $A_1^{1-\mu}u_n$, and dividing the obtained result by $|\lambda_n|^{-1}\lambda_n$, we get
\begin{equation}
-i\left\langle A_1^{1-\mu}u_n,v_n \right\rangle+\frac{1}{\lambda_n}|A_1^{1-\frac{\mu}{2}}u_n|^2 +\alpha\left \langle A_1^{1-\mu}u_n,\frac{1}{\lambda_n}B_1v_n \right\rangle+\beta\left \langle A_1^{1-\mu}u_n,\frac{1}{\lambda_n}B_1z_n \right\rangle   =o(1)\label{eqq1}.
\end{equation}

Since $A^{1-\mu}u_n$ and $v_n$ are bounded, the first term in the left hand side of (\ref{eqq1}) is bounded.

\medskip
 
On the other hand, from (\ref{eq1}) and the boundedness of $|A_1^{\frac{1}{2}}u_n|$, we deduce that $\frac{1}{\lambda_n}A_1^{\frac{1}{2}}v_n$ is bounded. 

The estimate of $\left \langle A_1^{1-\mu}u_n,\frac{1}{\lambda_n}B_1v_n \right\rangle $ in (\ref{eqq1}) will be proved as follows:
\begin{eqnarray}
\left| \left\langle A_1^{1-\mu}u_n,\frac{1}{\lambda_n} B_1 v_n\right\rangle \right|= \left| \left\langle \frac{1}{\lambda_n^{\frac{1}{2}}}B_1^{\frac{1}{2}}\left( A_1^{1-\mu}u_n\right) ,\frac{1}{\lambda_n^{\frac{1}{2}}}B_1^{\frac{1}{2}}v_n\right\rangle \right| \leq \frac{\left|B_1^{\frac{1}{2}}\left( A_1^{1-\mu}u_n\right)\right|}{|\lambda_n|^{\frac{1}{2}}}\frac{|B_1^{\frac{1}{2}}v_n|}{|\lambda_n|^{\frac{1}{2}}}\notag\\
\leq \frac{1}{4}\left( \frac{|A_1^{1-\frac{\mu}{2}}u_n|}{|\lambda_n|^{1\over2}}\right)^2 +C\left( \frac{|B_1^{1 \over2}v_n|}{|\lambda_n|^{1\over2}}\right)^2 ,\label{est0}
\end{eqnarray}
where we have used (\ref{b1}), and hereafter, $C$ denotes a generic positive constant that is independent of $n$.\\ Similarly, we give an estimate of $\left \langle A_1^{1-\mu}u_n,\frac{1}{\lambda_n}B_1z_n \right\rangle $  as follows:
\begin{eqnarray*}
\left| \left\langle A_1^{1-\mu}u_n,\frac{1}{\lambda_n}B_1 z_n\right\rangle \right|= \left| \left\langle \frac{1}{\lambda_n^{\frac{1}{2}}}B_1^{\frac{1}{2}}\left( A_1^{1-\mu}u_n\right) ,\frac{1}{\lambda_n^{\frac{1}{2}}}B_1^{\frac{1}{2}}z_n\right\rangle \right| \leq \frac{|B_1^{\frac{1}{2}}\left( A_1^{1-\mu}u_n\right)|}{|\lambda_n|^{\frac{1}{2}}}\frac{|B_1^{\frac{1}{2}}z_n|}{|\lambda_n|^{\frac{1}{2}}}\\
\leq \frac{1}{4}\left( \frac{|A_1^{1-\frac{\mu}{2}}u_n|}{|\lambda_n|^{1\over2}}\right)^2 +C\left( \frac{|B_1^{1 \over2}z_n|}{|\lambda_n|^{1\over2}}\right) ^2.
\end{eqnarray*}
Furthermore, recall that $\frac{|B_1^{1 \over2}v_n|}{|\lambda_n|^{1\over2}}=o(1)$ and $\frac{|B_1^{1 \over2}z_n|}{|\lambda_n|^{1\over2}}=o(1)$ (by (\ref{diss})). All that leads to the boundedness of $\frac{|A_1^{1-\frac{\mu}{2}}u_n|}{|\lambda_n|^{\frac{1}{2}}}$.
\\
Returning to (\ref{csch1}), we get, (keeping in mind again $\frac{|A_1^{\mu \over2}v_n|}{\lambda_n^{1\over2}}=o(1)$, from (\ref{diss}) and (\ref{b1})), 
\begin{equation}
\frac{1}{\lambda_n} \left\langle A_1u_n,v_n\right\rangle =o(1).
\label{t1}
\end{equation}
Similarly, 
we have
\begin{equation}
\left|\frac{1}{\lambda_n} \left\langle A_2w_n,z_n\right\rangle \right| \leq \frac{|A_2^{1-\frac{\theta}{2}}w_n|}{|\lambda_n|^{1\over2}}\frac{|A_2^{\theta \over2}z_n|}{|\lambda_n|^{1\over2}}.
\label{csch2}
\end{equation}
Taking the inner product of (\ref{eq4}) with $A_2^{1-\theta}w_n$, which is also bounded, since $\theta \geq \frac{1}{2}$, we get after dividing by  $|\lambda_n|^{-1}\lambda_n$

\begin{equation}
-i\left\langle A_2^{1-\theta}w_n,z_n \right\rangle+\frac{1}{\lambda_n} |A_2^{1-\frac{\theta}{2}}w_n|^2 +\alpha\left \langle A_2^{1-\theta}w_n,\frac{1}{\lambda_n}B_1 v_n \right\rangle+\gamma\left \langle A_2^{1-\theta}w_n,\frac{1}{\lambda_n}B_2 z_n \right\rangle  =o(1).\label{eq'''}
\end{equation}
As in (\ref{eqq1}), the first term  in (\ref{eq'''}) is bounded. Moreover, as in (\ref{est0}), using (\ref{b12}) and (\ref{b2}), we have the following estimate of $\left \langle A_2^{1-\theta}w_n,\frac{1}{\lambda_n}B_1v_n \right\rangle $ :

\begin{eqnarray*}
\left| \left\langle A_2^{1-\theta}w_n,\frac{1}{\lambda_n} B_1 v_n\right\rangle \right|= \left| \left\langle \frac{1}{\lambda_n^{\frac{1}{2}}}B_1^{\frac{1}{2}}\left( A_2^{1-\theta}w_n\right) ,\frac{1}{\lambda_n^{\frac{1}{2}}}B_1^{\frac{1}{2}}v_n\right\rangle \right| \leq \frac{\left|B_1^{\frac{1}{2}}\left( A_2^{1-\theta}u_n\right)\right|}{|\lambda_n|^{\frac{1}{2}}}\frac{\left|B_1^{\frac{1}{2}}v_n\right|}{|\lambda_n|^{\frac{1}{2}}}\\
\leq C\frac{| A_2^{1-\frac{\theta}{2}}w_n|}{|\lambda_n|^{\frac{1}{2}}}\frac{|B_1^{\frac{1}{2}}v_n|}{|\lambda_n|^{\frac{1}{2}}}\\
\leq \frac{1}{4}\left( \frac{|A_2^{1-\frac{\theta}{2}}w_n|}{|\lambda_n|^{1\over2}}\right)^2 +C\left( \frac{|B_1^{1 \over2}v_n|}{|\lambda_n|^{1\over2}}\right)^2.
\end{eqnarray*}
Similarly, we have 
\begin{align*}
\left| \left\langle A_2^{1-\theta}w_n,\frac{1}{\lambda_n} B_2 z_n\right\rangle \right|
\leq \frac{1}{4}\left( \frac{|A_2^{1-\frac{\theta}{2}}w_n|}{|\lambda_n|^{1\over2}}\right)^2 +C\left( \frac{|B_2^{1 \over2}z_n|}{|\lambda_n|^{1\over2}}\right)^2.
\end{align*}
Then, using that $\frac{|B_1^{1 \over2}v_n|}{|\lambda_n|^{\frac{1}{2}}}=o(1)$ and $\frac{|B_2^{1 \over2}z_n|}{|\lambda_n|^{\frac{1}{2}}}=o(1)$, (\ref{eq'''})  implies that $\frac{|A_2^{1-\frac{\theta}{2}}w_n|}{|\lambda_n|^{\frac{1}{2}}}=O(1)$. Keeping in mind $\frac{|A_2^{\theta \over2}z_n|}{\lambda_n^{1\over2}}=o(1)$ (from (\ref{diss}) and (\ref{b2})), it follows 
\begin{equation*}
\frac{1}{\lambda_n} \left\langle A_2w_n,z_n\right\rangle =o(1).\label{t2}
\end{equation*}
Reporting (\ref{t1}) and (\ref{csch2}) in (\ref{eQQ1}) and (\ref{eQQ2}) respectively to derive
$$|v_n|^2+|z_n|^2=o(1)$$
which contradicts  (\ref{eqq}).
\end{proof}

\section{Gevrey semigroup: case $\min(\mu,\theta)\in(0,\frac{1}{2})$}
\begin{thm} \label{th2}
For every $\mu,\theta\in (0,1]$ such that $s:=\min(\mu,\theta)\in(0,\frac{1}{2})$, the associated semigroup is of Gevrey class $\delta$ for every $\delta\geq \frac{1}{2s}$. More precisely, there exists a positive constant $C$ such that we have the resolvent estimate:
\begin{equation}
|\lambda|^{2s}\|(i\lambda I-\mathcal{A}_{\mu,\theta})^{-1}\|_{\mathcal{L}(\mathcal{H})}\leq  C,\;\;\;\forall\,\lambda \in \mathbb{R}.\label{sec6}
\end{equation}
\end{thm}
\begin{proof}
Suppose the result is false, then there exist a sequence $\lambda_n$ of real numbers where $|\lambda_n|$ goes to $\infty$ as $n\rightarrow \infty$, and a sequence $Z_n=(u_n,v_n,w_n,z_n)$ in $D(\mathcal{A}_{\mu, \theta})$, with $||Z_n||=1$ such that
\begin{equation}\label{lim3}
\lim_{|\lambda_n|\to\infty} |\lambda_n|^{-2s}\left\vert(i\lambda I - {\cal A}_{\mu,\theta})Z_n\right\vert=0.
\end{equation}
Taking  the real part of the inner product of $|\lambda_n|^{-2s}\left(i\lambda_n I - {\cal A}_{\mu,\theta}\right) Z_n$ with $Z_n$, we derive from the dissipativity estimate:
\begin{equation}\label{dissG}
|\lambda_n|^{-s}|B_1^{\frac{1}{2}}v_n|=o(1),\;\;\;\;\;\;|\lambda_n|^{-s}|B_2^\frac{1}{2}z_n|=o(1)\;\;\;\; \text{and}\;\;\;\;|\lambda_n|^{-s}|B_1^\frac{1}{2}z_n|=o(1).
\end{equation}
Equation \eqref{lim3} may be rewritten as:
\begin{eqnarray}
|\lambda_n|^{-2s}\left(i\lambda_nA_1^{1\over2}u_n-A_1^{1\over2}v_n\right)=o(1), \label{eq1G}\\
|\lambda_n|^{-2s}\left(i\lambda_nv_n+A_1u_n+\alpha B_1 v_n + \beta B_1 z_n\right):=h_n=o(1),\label{eq2G}\\
|\lambda_n|^{-2s}\left(i\lambda_nA_2^{1\over2}w_n-A_2^{1\over2}z_n\right)=o(1),\label{eq3G}\\
|\lambda_n|^{-2s}\left(i\lambda_nz_n+ \beta B_1v_n+ A_2w_n+\gamma B_2 z_n\right):=k_n=o(1).\label{eq4G}
\end{eqnarray}
As in the proof of the last theorem, one can deduce:

\begin{eqnarray}
i|v_n|^2+\frac{1}{\lambda_n}\left\langle A_1u_n,v_n\right\rangle
  =o(1),\label{eq211G}\\
 i|z_n|^2+\frac{1}{\lambda_n}\left\langle A_2w_n,z_n\right\rangle  =o(1)\label{eq411G},
\end{eqnarray}
and
\begin{eqnarray}
 |v_n|^2+|z_n|^2-\frac{1}{2} =o(1).\label{eqqG}
\end{eqnarray}
Again as in the first case, we will prove that $|v_n|=o(1)$ and $|z_n|=o(1)$ which would contradict (\ref{eqqG}).  To this end, we shall prove that $\frac{1}{\lambda_n} \left\langle A_1u_n,v_n\right\rangle =o(1)$ and $\frac{1}{\lambda_n} \left\langle A_2w_n,z_n\right\rangle =o(1)$.\\
The application of the Cauchy-Schwarz inequality yields
\begin{equation}
|\frac{1}{\lambda_n} \left\langle A_1u_n,v_n\right\rangle|\leq \frac{|A_1^{1-\frac{\mu}{2}}u_n|}{|\lambda_n|^{1-s}}\frac{|A_1^{\mu \over2}v_n|}{|\lambda_n|^{s}}.\label{csch3}
\end{equation}
The major  difficulty is to prove that $\frac{|A_1^{1-\frac{\mu}{2}}u_n|}{|\lambda_n|^{1-s}}$ is bounded.
Taking the inner product of (\ref{eq2G}) with $|\lambda_n|^{2s} \lambda_n^{-1}\frac{1}{\lambda_n^{1-2s}} A_1^{1-\mu}u_n$ to get
\begin{eqnarray*}
\frac{1}{\lambda_n^{1-2s}}\left\langle A_1^{1-\mu}u_n,iv_n \right\rangle+\frac{1}{\lambda_n^{2-2s}} |A_1^{1-\frac{\mu}{2}}u_n|^2 \\
+\left \langle \frac{1}{\lambda_n^{1-s}} A_1^{1-\mu}u_n,\frac{1}{\lambda_n^{1-s}}\left( \alpha B_1 v_n+ \beta B_1 z_n\right)  \right\rangle  =\frac{1}{\lambda_n^{1-2s}}\left\langle A_1^{1-\mu}u_n,|\lambda_n|^{2s} \lambda_n^{-1}h_n \right\rangle.
\end{eqnarray*}
Consequently, it follows
\begin{eqnarray}
\frac{1}{\lambda_n^{2-2s}}|A_1^{1-\frac{\mu}{2}}u_n|^2=-\left \langle \frac{1}{\lambda_n^{1-s}} A_1^{1-\mu}u_n,\frac{1}{\lambda_n^{1-s}}\left( \alpha B_1 v_n+ \beta B_1 z_n\right)  \right\rangle  \notag \\
+\left\langle \frac{1}{\lambda_n^{1-2s}}A_1^{1-\mu
}u_n,(|\lambda_n|^{2s} \lambda_n^{-1}h_n-iv_n) \right\rangle.\label{sec2'}
\end{eqnarray}
First, we have the following estimate

\begin{eqnarray*}
\left|\left \langle \frac{1}{\lambda_n^{1-s}} A_1^{1-\mu}u_n,\frac{1}{\lambda_n^{1-s}}\left( \alpha B_1 v_n+ \beta B_1 z_n\right)  \right\rangle  \right| \leq \frac{\left|B_1^{\frac{1}{2}}\left( A_1^{1-\mu}u_n\right) \right|}{|\lambda_n|^{1-s}}\frac{\alpha |B_1^{\frac{1}{2}} v_n|+|\beta| |B_1^{\frac{1}{2}} z_n|}{|\lambda_n|^{1-s}}\\
\leq  C\frac{|A_1^{1-\frac{\mu}{2}}u_n|}{|\lambda_n|^{1-s}}\frac{\left( |B_1^{\frac{1}{2}} v_n|+|B_1^{\frac{1}{2}} z_n|\right)}{|\lambda_n|^s},
\end{eqnarray*}
where we have used (\ref{b1}), and the elementary inequality:  $1-s > s$, since $s\in (0,\frac{1}{2})$.

Using Young inequality and the estimates
\begin{equation*}
|\lambda_n|^{-s}|B_1^{\frac{1}{2}}v_n|=o(1)\;\;\;\; \text{and}\;\;\;\;||\lambda_n|^{-s}|B_1^\frac{1}{2}z_n|=o(1)
\end{equation*}
 we deduce 
\begin{equation}
\left|\left \langle \frac{1}{\lambda_n^{1-s}} A_1^{1-\mu}u_n,\frac{1}{\lambda_n^{1-s}}\left( \alpha B_1 v_n+ \beta B_1 z_n\right)  \right\rangle  \right|\leq \frac{1}{3}\left( \frac{|A_1^{1-\frac{\mu}{2}}u_n|}{|\lambda_n|^{1-s}}\right)^2+o(1).\label{4es1}
\end{equation}
We are going to estimate the last term in (\ref{sec2'}). First, note that if $\mu \geq \frac{1}{2}$, then  $\frac{1}{\lambda_n^{1-2s}}A_1^{1-\mu}u_n$ is bounded (because $A_1^{1/2}u_n$ is bounded and $s<\frac{1}{2}$), so that $\left\langle \frac{1}{\lambda_n^{1-2s}}A_1^{1-\mu}u_n,(|\lambda_n|^{2s} \lambda_n^{-1}h_n-iv_n) \right\rangle$ is bounded. Next, let us  assume  that $0<\mu<\frac{1}{2}$. We have
 the following two estimates
\begin{equation*}
\left| \left\langle \frac{1}{\lambda_n^{1-2s}}A_1^{1-\mu}u_n,(|\lambda_n|^{2s} \lambda_n^{-1}h_n-iv_n) \right\rangle \right| \leq \frac{|A_1^{1-\mu}u_n|}{|\lambda_n|^{1-2s}}(|h_n|+|v_n|), 
\end{equation*}
the second one is obtained by applying an interpolation inequality:
\begin{equation*}
\frac{|A_1^{1-\mu}u_n|}{|\lambda_n|^{1-2s}} \leq C  \left(\frac{|A_1^{1-\frac{\mu}{2}}u_n|}{|\lambda_n|^{\frac{(1-\mu)(1-2s)}{1-2\mu}}}\right)^{\frac{1-2\mu}{1-\mu}} |A_1^{\frac{1}{2}}u_n|^{\frac{\mu}{1-\mu}}.
\end{equation*}
The combination of those two estimates leads to the following estimate, (keeping in mind that $\frac{(1-\mu)(1-2s)}{1-2\mu} \geq 1-s$):

\begin{equation*}
\left| \left\langle \frac{1}{\lambda_n^{1-2s}}A_1^{1-\mu}u_n,(|\lambda_n|^{2s} \lambda_n^{-1}|\lambda_n|^{2s} \lambda_n^{-1}h_n-iv_n) \right\rangle \right| \leq C\left(\frac{|A_1^{1-\frac{\mu}{2}}u_n|}{|\lambda_n|^{1-s}}\right)^{\frac{1-2\mu}{1-\mu}} |A_1^{\frac{1}{2}}u_n|^{\frac{\mu}{1-\mu}}(|h_n|+|v_n|).
\end{equation*}
Again, using Young inequality and the boundedness of  $|A_1^{\frac{1}{2}}u_n|$, $|h_n|$ and $|v_n|$, we find
\begin{equation}
\left| \left\langle \frac{1}{\lambda_n^{1-2s}}A_1^{1-\mu}u_n,(|\lambda_n|^{2s} \lambda_n^{-1}h_n-iv_n) \right\rangle \right| \leq \frac{1}{3}\left( \frac{|A_1^{1-\frac{\mu}{2}}u_n|}{|\lambda_n|^{1-s}}\right)^2+C.\label{4es2}
\end{equation}
Now, using (\ref{4es1}) and (\ref{4es2}) in (\ref{sec2'}), we derive that $\left( \frac{|A_1^{1-\frac{\mu}{2}}u_n|}{|\lambda_n|^{1-s}}\right)^2$ is bounded. Finally, recall that  $\frac{|A_1^{\mu \over2}v_n|}{|\lambda_n|^{s}}=o(1)$ (by (\ref{dissG}) and (\ref{b1})), then we get, using (\ref{csch3}), the estimate
\begin{equation}
\frac{1}{\lambda_n} \left\langle A_1u_n,v_n\right\rangle =o(1).\label{tt1}
\end{equation} 
Similarly, we will prove that $\frac{1}{\lambda_n} \left\langle A_2w_n,z_n\right\rangle
=o(1).$  
We start by the following estimate
\begin{equation*}
\left|\frac{1}{\lambda_n} \left\langle A_2w_n,z_n\right\rangle \right| \leq \frac{|A_2^{1-\frac{\theta}{2}}w_n|}{|\lambda_n^{1-s}|}\frac{|A_2^{\theta \over2}z_n|}{|\lambda_n^{s}|}.\label{csch4}
\end{equation*}
As for (\ref{tt1}), it suffices to show that $\frac{|A_2^{1-\frac{\theta}{2}}w_n|}{|\lambda_n^{1-s}|}$ is bounded (since $\frac{|A_2^{\theta \over2}z_n|}{|\lambda_n|^{s}}=o(1)$ by (\ref{dissG}) and (\ref{b2}) ). \\
Taking the inner product of (\ref{eq2G}) with $|\lambda_n|^{2s} \lambda_n^{-1}\frac{1}{\lambda_n^{1-2s}} A_2^{1-\theta}w_n$ we deduce that
\begin{eqnarray}
\frac{1}{\lambda_n^{2-2s}}|A_2^{1-\frac{\theta}{2}}w_n|^2=-\left \langle \frac{1}{\lambda_n^{1-s}} A_2^{1-\theta}w_n,\frac{1}{\lambda_n^{1-s}} \alpha B_1 v_n  \right\rangle -\left \langle \frac{1}{\lambda_n^{1-s}} A_2^{1-\theta}w_n,\frac{1}{\lambda_n^{1-s}}\beta  B_2
 z_n  \right\rangle\notag \\
+\left\langle \frac{1}{\lambda_n^{1-2s}}A_2^{1-\theta
}u_n,(|\lambda_n|^{2s} \lambda_n^{-1}k_n-iz_n) \right\rangle.\label{sec2''}
\end{eqnarray}
First, we have the following two estimates
\begin{eqnarray}
\left|\left \langle \frac{1}{\lambda_n^{1-s}} A_2^{1-\theta}w_n,\frac{1}{\lambda_n^{1-s}} \alpha B_1 v_n  \right\rangle  \right|= \left|\left \langle \frac{1}{\lambda_n^{1-s}} B_1^{1 \over2}\left( A_2^{1-\theta}w_n\right) ,\frac{1}{\lambda_n^{1-s}} \alpha B_1^{1 \over2} v_n  \right\rangle  \right|\notag\\
\leq C \frac{|A_2^{1-\frac{\theta}{2}}w_n|}{|\lambda_n|^{1-s}}.\frac{ |B_1^{\frac{1}{2}} v_n|}{|\lambda_n|^{1-s}}\label{est1}
\end{eqnarray}
and
\begin{eqnarray}
\left|\left \langle \frac{1}{\lambda_n^{1-s}} A_2^{1-\theta}w_n,\frac{1}{\lambda_n^{1-s}} \beta B_2 z_n  \right\rangle  \right|= \left|\left \langle \frac{1}{\lambda_n^{1-s}} B_2^{1 \over2}\left( A_2^{1-\theta}w_n\right) ,\frac{1}{\lambda_n^{1-s}} \beta B_2^{1 \over2} z_n  \right\rangle  \right|\notag\\
\leq C \frac{|A_2^{1-\frac{\theta}{2}}w_n|}{|\lambda_n|^{1-s}}.\frac{ |B_2^{\frac{1}{2}} z_n|}{|\lambda_n|^{s}},\label{est2}
\end{eqnarray}
where we have used that $1-s > s$, since $s\in (0,\frac{1}{2})$.\\
To estimate the last term in (\ref{sec2''}) we proceed as in (\ref{sec2'}): if $\theta \geq \frac{1}{2}$, then  $\frac{1}{\lambda_n^{1-2s}}A_2^{1-\theta}w_n$ is bounded, so that $\left\langle \frac{1}{\lambda_n^{1-2s}}A_2^{1-\theta}w_n,(|\lambda_n|^{2s} \lambda_n^{-1}k_n-iz_n) \right\rangle$ is bounded. If $0<\theta<\frac{1}{2}$: we have
 the following two estimates: the first is 
\begin{equation*}
\left| \left\langle \frac{1}{\lambda_n^{1-2s}}A_2^{1-\theta}w_n,(|\lambda_n|^{2s} \lambda_n^{-1}k_n-iz_n) \right\rangle \right| \leq \frac{|A_2^{1-\theta}w_n|}{|\lambda_n|^{1-2s}}(|k_n|+|z_n|), 
\end{equation*}
the second, is obtained by applying an interpolation inequality:
\begin{equation*}
\frac{|A_2^{1-\theta}w_n|}{|\lambda_n|^{1-2s}} \leq C \left(\frac{|A_2^{1-\frac{\theta}{2}}w_n|}{|\lambda_n|^{\frac{(1-\theta)(1-2s)}{1-2\theta}}}\right)^{\frac{1-2\theta}{1-\theta}} |A_2^{\frac{1}{2}}w_n|^{\frac{\theta}{1-\theta}},
\end{equation*}
to get, (keeping in mind that $\frac{(1-\theta)(1-2s)}{1-2\theta} \geq 1-s$),
the following estimate
\begin{equation}
\left| \left\langle \frac{1}{\lambda_n^{1-2s}}A_2^{1-\theta}w_n,(|\lambda_n|^{2s} \lambda_n^{-1}k_n-iz_n) \right\rangle \right| \leq C\left(\frac{|A_2^{1-\frac{\theta}{2}}w_n|}{|\lambda_n|^{1-s}}\right)^{\frac{1-2\theta}{1-\theta}} |A_2^{\frac{1}{2}}w_n|^{\frac{\theta}{1-\theta}}(|k_n|+|z_n|).\label{est3}
\end{equation}
Again, using appropriate Young inequality at (\ref{est1},\ref{est2}) and (\ref{est3}) we deduces from (\ref{sec2''}) the boundedness of  $\left( \frac{|A_2^{1-\frac{\theta}{2}}w_n|}{|\lambda_n|^{1-s}}\right)^2$ is bounded. Finally, recall that  $\frac{|A_2^{\theta \over2}z_n|}{|\lambda_n|^{s}}=o(1)$, then we get by (\ref{csch4}), the estimate
\begin{equation}
\frac{1}{\lambda_n} \left\langle A_2w_n,z_n\right\rangle =o(1). 
\end{equation}
\end{proof}
\section{An optimal result}

In this section we suppose that $A_1=A_2:=A$, $B_1:=A^{\mu}$, $B_2:=A^{\theta}$, and instead of  (\ref{b12}), we take $\mu \leq \theta$, which implies (\ref{b12}).
\begin{thm}
The resolvent estimate (\ref{sec6}) is optimal, in the sense that, for every $\mu,\theta \in (0,1]$ such that $\mu\in(0,\frac{1}{2})$ and $\mu\leq \theta$, for every  $r\in (2\mu, 1]$, we have:
\begin{equation}
\limsup_{|\lambda|\rightarrow\infty}|\lambda|^{r}\|(i\lambda I-\mathcal{A}_{\mu,\theta})^{-1}\|_{\mathcal{L}(\mathcal{H})}=\infty.\label{sec7}
\end{equation}
\end{thm}
\begin{proof} 
We are going to show that there exist a sequence of  positive real numbers  $(\lambda_n)_{n\geq1}$, and for each $n$, an element $Z_n\in {\cal D}({\cal A})$
such that for every $r\in(2\mu,1]$, one has:
\begin{align}\label{op1}
\lim_{n\to\infty}\lambda_n=\infty,\quad ||Z_n||=1, \quad \lim_{n\to\infty}\lambda_n^{-r}||(i\lambda_n-{\cal A}_{\mu,\theta})Z_n||=0.
\end{align}
Indeed, if we have sequences $\lambda_n$ and $Z_n$ satisfying \eqref{op1}, then we set
\begin{align}\label{op2}
V_n=\lambda_n^{-r}(i\lambda_n-{\cal A}_{\mu,\theta})Z_n,\qquad U_n=\frac{V_n}{||V_n||}.
\end{align}
 Therefore,
$||U_n||=1$ and
\begin{align}\label{op3}
\lim_{n\to\infty}\lambda_n^r||(i\lambda_n-{\cal A}_{\mu,\theta})^{-1}U_n||=\lim_{n\to\infty}\frac{1}{||V_n||}=\infty,
\end{align}
which would establish the claimed result.
 Thus, it remains to prove the existence of such sequences.

\medskip

 For each
$n\geq1$, let $e_n$ be the eigenfunction of the operator $A$, and $\omega_n$ be its corresponding eigenvalue as in the proof of \cite[Theorem 1.2]{kfl}. As in that proof, we seek $Z_n$ in
the form $Z_n=(a_ne_n,i\lambda_na_ne_n,c_ne_n,i\lambda_nc_ne_n)$, with $\lambda_n$ and the complex numbers
$a_n$ and $c_n$ chosen such that $Z_n$ fulfills the
desired conditions.

\medskip

Set
\begin{align}\label{op4}
\lambda_n=\sqrt { \omega_n}.
\end{align}
With that choice, we readily check that:
\begin{equation}\label{op5}\begin{array}{lll}
(i\lambda_n-{\cal A}_{\mu,\theta})Z_n&= \begin{pmatrix}
   0\\
\left[(\omega_n-\lambda_n^2)a_n+i\lambda_n\omega_n^{\mu}(\alpha a_n+\beta c_n)\right]e_n\\
0\\ (-\lambda_n^2+\omega_n)c_n+i\lambda_n(\beta \omega_n^\mu a_n+\gamma w_n^\theta c_n))
e_n
 \end{pmatrix}\\
&= \begin{pmatrix}
   0\\
i\omega_n^{\mu+{1\over2}}\left[\alpha a_n+\beta c_n\right]e_n\\
0\\ \left(i\gamma\omega_n^{\theta-{1\over2}} c_n+i\beta\omega_n^{\mu-{1\over2}} a_n\right)\omega_ne_n
 \end{pmatrix},\text{ by }\eqre{op4}.\\
\end{array}\end{equation}
For each $n\geq1$, we set
\begin{align}\label{op6}
a_n=-\gamma\beta^{-1}\omega_n^{\theta-\mu}c_n.\end{align}
 It then follows from \eqref{op5}:

 \begin{equation}\label{op7}\begin{array}{lll}
(i\lambda_n-{\cal A}_{\mu,\theta})Z_n&= \begin{pmatrix}
   0\\
\left(-\alpha\gamma\beta^{-1}w_n^{\theta-\mu}+\beta \right) i\omega_n^{\mu+{1\over2}}c_ne_n\\0\\0
 \end{pmatrix} 
\end{array}
\end{equation}
 There exists $k>0$ such that
\begin{align}\label{op8}
&\lim_{n\to\infty}\lambda_n^{-2r}||(i\lambda_n-{\cal A}) Z_n||^2\notag\\
&=\lim_{n\to\infty}k^2|c_n|^2\lambda_n^{-2r}\omega_n^{2\theta+1} |e_n|^2 =k^2\lim_{n\to\infty}\omega_n^{2\mu-r}\omega_n^{2\theta-2\mu+1}|c_n|^2\\
&=0,\text{ for } r>2\mu, \text{ and }\notag
\end{align}provided that the sequence $(\omega_n^{2\theta-2\mu+1}|c_n|^2)$ converges to  some nonzero real number, and $||Z_n||=1$.\\
 One checks that
\begin{align}\label{op9}
||Z_n||^2&=\omega_n|a_n|^2+\lambda_n^2|a_n|^2+\omega_n|c_n|^2+\lambda_n^2|c_n|^2
\notag\\
&=2\omega_n|a_n|^2+2\omega_n|c_n|^2\notag\\
&=2\gamma^2\beta^{-2}w_n^{2\theta-2\mu}\omega_n|c_n|^2+2\omega_n|c_n|^2\\
&=2\left(\gamma^2\beta^{-2}+w_n^{-2(\theta-\mu)}\right)\omega_n^{2\theta-2\mu+1}|c_n|^2,\notag\end{align} 
so that we might just choose
\begin{align}\label{op10}c_n={\omega_n^{ \mu-\theta-\frac{1}{2}} \over\sqrt{2\left( \gamma^2\beta^{-2}+w_n^{2\mu-2\theta}\right) }}\end{align} to get $||Z_n||=1$ as desired.
\end{proof}

\section{Examples of application}

Let $\Omega$ be a bounded domain in ${\bb R}^N$ with smooth boundary $\Gamma$. Typical examples of application include, but are not limited to
\begin{enumerate}
\item {\bf Interacting membranes}
\begin{equation*}\begin{array}{ll}
&y_{tt}-a\Delta y+\alpha(-\Delta)^\mu y_t+\beta(-\Delta)^\mu z_t=0\text{ in }\Omega\times(0,\infty)\\
&z_{tt}-b\Delta z+\beta(-\Delta)^\mu y_t+\gamma(-\Delta)^\theta z_t=0\text{ in }\Omega\times(0,\infty)\\
&y=0,\quad z=0\text{ on }\Gamma\times(0,\infty),
\end{array}\end{equation*}

where $\alpha, \beta, \gamma, \mu$ and $\theta$ are as in Section \ref{sec1}.

\medskip

Here, $H=L^2(\Omega)$, $A_1=A_2:=A=-\Delta, B_1 = (-\Delta)^\mu, B_2 = (-\Delta)^\theta$ with $D(A)= H^2(\Omega) \cap H^1_0 (\Omega)$. Then $A$ is a densely defined, positive unbounded operator on the Hilbert space $H$. Moreover, $V:=D(A^{1 \over2})=H_0^1 (\Omega)$ and the injections $V \hookrightarrow H \hookrightarrow V^\prime=H^{-1}(\Omega)$ are dense and compact.

\medskip

Then the corresponding semigroup is 

\begin{itemize}
\item
analytic, according to Theorem \ref{th1}, for $1/2 \leq \mu,\theta \leq 1.$
\item
in Gevrey class $\delta > \frac{1}{2\mu}$, according to Theorem \ref{th2}, for $\mu\in (0,1/2), \theta \in (0,1]$ and $\mu \leq \theta$. 
\end{itemize}

\item{\bf Interacting membrane and plate}
\begin{equation*}\begin{array}{ll}
&y_{tt}-\Delta y+\alpha(-\Delta)^\mu y_t+\beta(-\Delta)^\mu z_t=0\text{ in }\Omega\times(0,\infty)\\
&z_{tt}+\Delta^2 z+\beta(-\Delta)^\mu y_t+\gamma\Delta^{2\theta} z_t=0\text{ in }\Omega\times(0,\infty)\\
&y=0,\quad z=0,\quad \frac{\partial z}{\partial\nu}=0\text{ on }\Gamma\times(0,\infty),
\end{array}\end{equation*}

where $\alpha, \beta, \gamma, \mu$ and $\theta$ as in Section \ref{sec1}. 

\medskip

Here, $H=L^2(\Omega)$, $A_1=-\Delta, A_2 = \Delta^2, B_1 = (-\Delta)^\mu, B_2 = \Delta^{2\theta}$ with $D(A_1)= H^2(\Omega) \cap H^1_0 (\Omega), D(A_2) = H^4(\Omega) \cap H^2_0(\Omega)$. Then $A_i, i=1,2,$ are densely defined, positive unbounded operators on the Hilbert space $H$. Moreover, $V_1:=D(A_1^{1 \over2})=H_0^1 (\Omega), V_2 = H^2_0 (\Omega)$ and the injections $V_i \hookrightarrow H \hookrightarrow V_i^\prime, V_1^\prime = H^{-1}(\Omega), V_2=H^{-2}(\Omega),$ are dense and compact.

\medskip

Then the corresponding semigroup is 

\begin{itemize}
\item
analytic, according to Theorem \ref{th1}, for $1/2 \leq \mu, \theta \leq 1.$
\item
in Gevrey class $\delta > \frac{1}{2\min(\mu,\theta)}$, according to Theorem \ref{th2}, for $\mu, \theta \in (0,1]$ and $\min(\mu,\theta) \in(0,1/2).$ 
\end{itemize}

\item{\bf Interacting  plates}
\begin{equation*}\begin{array}{ll}
&y_{tt}+\Delta^2 y+\alpha\Delta^{2\mu} y_t+\beta\Delta^{2\mu} z_t=0\text{ in }\Omega\times(0,\infty)\\
&z_{tt}+\Delta^2 z+\beta\Delta^{2\mu} y_t+\gamma\Delta^{2\theta} z_t=0\text{ in }\Omega\times(0,\infty)\\
&y=0,\quad \frac{\partial y}{\partial\nu}=0,\quad z=0,\quad \Delta z=0\text{ on }\Gamma\times(0,\infty),
\end{array}\end{equation*}
\end{enumerate}

where $\alpha, \beta, \gamma, \mu$ and $\theta$ are as in Section \ref{sec1}. 

\medskip

Here, $H=L^2(\Omega)$, $A_1=\Delta^2$ with $D(A_1)= H^4(\Omega) \cap H^2_0 (\Omega),$ and $ B_1 = A_1^{\mu},$ while $A_2=\Delta^2$ with $D(A_2)=\{u\in H^4(\Omega); u=\Delta u=0\text{ on }\Gamma\}$, and $ B_2 = A_2^{\theta}$. The operators $A_1$, $A_2$, $B_1$ and $B_2$ satisfy all the desired requirements.
\medskip

\noindent
The the corresponding semigroup is 

\begin{itemize}
\item
analytic, according to Theorem \ref{th1}, for $1/2 \leq \mu \leq \theta \leq 1.$
\item
in Gevrey class $\delta > \frac{1}{2\mu}$, according to Theorem \ref{th2}, for $\mu\in (0,1/2), \theta \in (0,1]$ and $\mu \leq \theta$. 
\end{itemize}


\begin{thebibliography}{99}

\bibitem{kfl} {K. Ammari, F. Shel and L. Tebou, Regularity and stability of the semigroup associated with some interacting elastic systems I: A degenerate damping case,  J. Evol. Equ.  (2021). https://doi.org/10.1007/s00028-021-00738-7.}

\bibitem{ArBa} {W. Arendt and C. J. K. Batty, Tauberian theorems and stability of one-parameter semigroups, Trans. Amer.
Math. Soc., 306 (1988), 837-852.}

\bibitem{cru} {G. Chen and D. L. Russell, A mathematical model for linear elastic systems with structural
damping, Quart. Appl. Math. 39 (1982), 433-454.}

\bibitem{ctra} {S. P. Chen and R. Triggiani,  Proof of extensions of two conjectures on structural damping for
 elastic systems. Pacific J. Math. 136 (1989),  15-55.}

\bibitem{ctrg}
{S. P. Chen and R. Triggiani, Gevrey class semigroups arising from elastic systems 
with gentle dissipation: the case $0 <\alpha < 1/2$. Proc. Am. Math. Soc. 110 (1990), 401-415.}

\bibitem{H2}
{F. Huang,  On the mathematical model for linear elastic systems
with analytic damping,  SIAM J. Control Optim.,  26
(1988), 714-724.}

\bibitem{HuL}
F. Huang and K. Liu,  {Holomorphic property and exponential
stability of the semigroup associated with linear elastic systems
with damping,  Ann. Diff. Eqs.,  4(1988), 411-424.}

\bibitem{ly} {Z. Liu and J.  Yong,  Qualitative properties of certain $C_0$ semigroups arising in elastic systems with various dampings. Adv. Differential Equations 3 (1998),  643-686.}

\bibitem{LiZh} {Z. Liu and S. Zheng, Semigroups associated with dissipative systems, Chapman and Hall/CRC, 1999.}

\bibitem{taylor} {S. Taylor, Gevrey regularity of solutions of evolution equations and boundary controllability, Gevrey semigroups (Chapter 5), Ph.D Thesis, School of Mathematics, University of Minnesota, 1989.}

\end{thebibliography}
\end{document}